\documentclass[12pt,a4paper,twoside]{amsart}
\usepackage{amsfonts, amsthm, amsmath, amssymb}
\usepackage{hyperref}
\hypersetup{colorlinks=false}

\usepackage[margin=1.25in]{geometry}

\usepackage{helvet}

\usepackage{amsthm}
\newtheorem{theorem}{Theorem}
\newtheorem{lemma}{Lemma}[section]
\newtheorem{remark}{Remark}

\newtheorem{conjecture}{Conjecture}[section]
\newtheorem{proposition}{Proposition}
\newtheorem{corollary}{Corollary}

\begin{document}

\author{Kummari Mallesham}
\title{On pairs of quadratic forms in five variables}

\address{ Kummari Mallesham \newline {\em Stat-Math Unit, Indian Statistical Institute, 203 B.T. Road, Kolkata 700108, India; \newline  Email: iitm.mallesham@gmail.com
} }

\maketitle

\begin{abstract}
In this article, we obtain an upper bound for the number of integral solutions, of given height, of system of two quadratic forms in five variables. Our bound is an improvement over the bound given by Henryk Iwaniec and Ritabrata Munshi in \cite{H-R}.
\end{abstract}

\section{Introduction}
Let $\phi_{1}(x_{1},\ldots,x_{5})$ and $\phi_{2}(x_{1},\ldots,x_{5})$ be two quadratic forms over $\mathbb{Z}$ in five variables. Then the system of quadratic equations 
\begin{equation*}
\phi_{1}(x_{1},\ldots,x_{5})=0, \phi_{2}(x_{1},\ldots,x_{5})=0
\end{equation*}
define a quartic del Pezzo surface ${\bf V} \subset \mathbb{P}^{4}$. It is very interesting question to understand the estimates for the counting function
\begin{equation}
N(B) = \vert \lbrace {\bf x} \in \mathbb{Z}^{5} : {\bf x} \, \, \text{is primitive}, \, {\bf x} \in U(\mathbb{Q}), \, \vert x_{i} \vert \leq B \rbrace\vert, 
\end{equation} 
where ${\bf x}$ primitive means gcd $(x_{1},\ldots,x_{5})=1$ and $x_{1}$ is positive, and $U$ be the open subset of ${\bf V}$ obtained by deleting lines from it. We have the following conjecture:
\begin{conjecture} \label{conj}
We have 
$$N(B) \sim c \, B \, \left (\log B \right)^{t-1},$$
where $c$ and $t$ are constants depending only on $\phi_{1}$ and $\phi_{2}$.
\end{conjecture}

This is a special case of a more general conjecture formulated by Batyrev, Franke, Manin, Tschinkel and others for Fano varieties. Here we are considering the case of del Pezzo surfaces, Fano varieties of dimension 2, of degree 4. For  more details of this conjecture we refer reader to \cite{Tim}.

Some progress has been made towards the conjecture \ref{conj} by T. D. Browning and De La Bret\`eche in \cite{Tim-Bret} when ${\bf V}$ contains isolated singularities. However, the conjecture is far being proved for non-singular del Pezzo surface of degree 4. One can easily obtain the bound $N(B) \ll B^{2+\epsilon}$. In general improving this bound is difficult. 

In \cite{H-R}, H. Iwaniec and R. Munshi obtained bounds of the form $N(B) \ll B^{2-\delta}$ for some absolute constant $\delta >0$ under certain mild conditions on the forms $\phi_{1}$ and $\phi_{2}$. Suppose the forms $\phi_{1}$ and $\phi_{2}$ are compatible in the sense that there is an unimodular  transformation that transforms $\phi_{1}$ and $\phi_{2}$ to 
$\psi_{1} \left({\bf x}\right) = a_{1} x_{1}^2+a_{2} x_{2}^2+a_{3}x^2+a_{4}x^2$
and $\psi^{\prime}_{2} \left({\bf x}\right) = b_{1} x_{1}^2+b_{2} x_{2}^2+b_{3}x_{3}^2+b_{4}x_{4}^2 -x_{5}^2$ respectively, with the condition that all  $2 \times2$ minors of the matrix 
\[
\begin{bmatrix}
a_{1} & a_{2} & a_{3} & a_{4} \\
b_{1} & b_{2} & b_{3} & b_{4} 
\end{bmatrix}
\]
are non-zero, where $a_{i}$'s and $b_{i}$'s are the coefficients appearing in the diagonal forms $\psi_{1}$ and $\psi^{\prime}_{2}$ respectively and $\alpha = a_{1}a_{2} a_{3} a_{4}$ is  not a square. We set $\mathcal{D}$ to be the products of all these minors. For ${\bf x} \in \mathbb{Z}^{4}$, we set $\psi_{2}({\bf x})= b_{1} x_{1}^2+b_{2} x_{2}^2+b_{3}x_{3}^2+b_{4}x_{4}^2 $. We note, in this case, that
$$N(B) \leq M(B),$$
where
$$M(B) = \vert \lbrace {\bf x} \in \mathbb{Z}^{4} : \, \psi_{1}({\bf x}) =0, \, \psi_{2}({\bf x})= \square, \, \vert x_{i} \vert  \leq B \rbrace\vert .$$

In \cite{H-R}, H. Iwaniec and R. Munshi obtained non-trivial bounds for the quantity $M(B)$ using the circle method and square sieve of Heath-Brown. More precisely, they have the following theorem.

\begin{theorem}[H. Iwaniec and R. Munshi, \cite{H-R}]
Under the condition on the forms $\psi_{1}$ and $\psi_{2}$ stated above, we have
\begin{equation} \label{hrbound}
M(B) \ll B^{\frac{9}{5}+ \epsilon},
\end{equation}
where the constant depends only on the forms $\psi_{1}, \psi_{2} $ and $\epsilon$.
\end{theorem} 

Our aim in this article is to improve the bound on $M(B)$ that has given in \eqref{hrbound}. We state our bound on $M(B)$ as a following theorem.
\begin{theorem} \label{malthe}
Under the same conditions on the forms $\psi_{1}$ and $\psi_{2}$ given above, we have that
\begin{equation} \label{malbound}
M(B) \ll B^{\frac{5}{3} + \epsilon},
\end{equation}
where the constant depends only on the forms $\psi_{1}, \psi_{2} $ and $\epsilon$.
\end{theorem}

We have the following:
\begin{corollary}
We have that
$$N(B) \ll_{\epsilon}  \, B^{5/3 +\epsilon}.$$
\end{corollary}

\begin{remark}
The best known results are due to P. Salberger \cite{Salberger}. He proves that $N(B) \ll_{\epsilon} B^{13/8 + \epsilon}$. However, our methods are are totally different from his methods. 
\end{remark}

As in \cite{H-R}, we consider the smoothed version of $M(B)$ and which is given by
\begin{equation} \label{countsol1}
M^{*}(B) = \sum_{\substack{{\bf x} \in \mathbb{Z}^4 \\ \psi_{1}({\bf x})=0}} W\left(B^{-1} {\bf x}\right) \, \theta \left(\psi_{2} \left({\bf x}\right)\right),
\end{equation}
where
\[ \theta(n) = 
\begin{cases} 
1 & \text{if} \, \,  n = \square \\
0 & \text{otherwise}.
\end{cases}
\] 
and $W: \mathbb{R}^4 \to \mathbb{R}$ be a non-negative smooth function supported on the cube $[-1,1]^4$. Clearly Theorem \ref{malthe} follows from the following proposition.
\begin{proposition} \label{malpropo}
We have 
$$M^{*}(B) \ll_{\epsilon} B^{\frac{5}{3}+\epsilon},$$
for any $\epsilon >0$.
\end{proposition}

The rest of the paper devoted for the proof of Proposition \ref{malpropo}. We also use circle method and square sieve of Heath-Brown, as in \cite{H-R}, to obtain an upper bound on $M^{*}(B)$. First we majorize the function $\theta (n)$ by using  square sieve, in this process we end up with character sum over a variety. In this character sum we get the cancellations using the circle method.  Before applying the delta symbol to detect an equation we use the `congruence-equation trick' to lower the conductor in the circle method. This trick was first used by T. D Browning and R. Munshi in \cite{Tim-Munshi} while studying rational points on singular intersections of quadrics and later R. Munshi used this trick to obtain sub-convexity bounds for higher degree $L$-functions, see \cite{munshi3} for example. The use of this trick here lead us to improve the bound for $M(B)$ given in \eqref{hrbound}. We give details of this discussion in the following sections.

\section{Majorant for $\theta(n)$}
Let $\mathcal{P}$ be a set of $P$ primes of size $P \log P$. Then for each $n \in \mathbb{Z}$ with $\vert n \vert < \exp{P}$, we have
\begin{equation} \label{squaredet}
\theta(n) \ll \frac{1}{P^2} \, \bigg \lvert \sum_{p \in \mathcal{P}} \chi_{p} \left(n\right) \bigg \rvert^2,
\end{equation}
where $\chi_{p}$ is the quadratic residue character modulo $p$. The construction of such a majorant for $\theta(n)$ is due to D. R. Heath-Brown and he calls it as square sieve.

Let $P$ be a parameter such that $P \geq c \log B$ for some constant $c$ depending only on the form $\psi_{2}$.  Using the majorant for $\theta$, given in \eqref{squaredet}, for this choice of $P$,  we see that 
\begin{align*}
M^{*}(B) & \ll \sum_{\substack{{\bf x} \in \mathbb{Z}^4 \\ \psi_{1}({\bf x})=0}} W\left(B^{-1} {\bf x}\right) \frac{1}{P^2} \, \bigg \lvert \sum_{p \in \mathcal{P}} \chi_{p} \left(n\right) \bigg \rvert^2 \\
& \ll \frac{1}{P^2} \sum_{p_{1},p_{2} \in \mathcal{P}} \sum_{\substack{{\bf x} \in \mathbb{Z}^4 \\ \psi_{1}({\bf x})=0}} W\left(B^{-1} {\bf x}\right) \,  \chi_{p_{1} p_{2}} \left(\psi_{2}\left({\bf x}\right)\right).
\end{align*}
The diagonal , that is when $p_{1} = p_{2}$,  contribution we estimate trivially using the crude bound 
$$\sum_{\substack{{\bf x} \in \mathbb{Z}^4 \\ \psi_{1}({\bf x})=0}} W\left(B^{-1} {\bf x}\right) \ll_{\epsilon} B^{2 + \epsilon}.$$
Therefore we have
\begin{equation} \label{mstarbound}
M^{*}(B) \ll_{\epsilon} \frac{B^{2+\epsilon}}{P} + \frac{1}{P^{2}} \sum_{p_{1} \neq p_{2} \in \mathcal{P}} T_{p_{1}p_{2}}\left(B\right),
\end{equation}
where
$$T_{q} (B)= \sum_{\substack{{\bf x} \in \mathbb{Z}^4 \\ \psi_{1}({\bf x})=0}} W\left(B^{-1} {\bf x}\right) \chi_{q} \left(\psi_{2}\left({\bf x}\right)\right).$$
We will exploit oscillations of the quadratic character $\chi_{q}$ to show some cancellation in the sum $T_{q} (B)$. This will be done using the circle method. More precisely,  we will prove the following proposition.
\begin{proposition} \label{propontq}
Let $q$ be a square free number. Then we have
\begin{equation}
T_{q}(B) \ll \left( \frac{B^2}{q^{3/2}} + q B + \frac{B^{3/2}}{q^{1/4}} \right) B^{\epsilon}.
\end{equation}
where implied constant depends only on the forms $\psi_{1}, \psi_{2}$ and $\epsilon$.
\end{proposition}  
We give proof of this proposition in Section \ref{tqbound}. Finally using the above bound on $T_{q}(B)$ in \eqref{mstarbound},  we get a bound on $M^{\star}(B)$. By choosing $P$ optimally,  we get a proof of Proposition \ref{malpropo}, which we discuss in the final Section \ref{final}.

\section{The $\delta$-symbol} 
The delta symbol $\delta : \mathbb{Z} \to \mathbb{R}$ is given by
\[ \delta(n) = 
\begin{cases} 
1 & \text{if} \, \,  n = 0 \\
0 & \text{otherwise}.
\end{cases}
\] 
We have a following analytic expression for this symbol which is due to D. R. Heath-Brown \cite[Theorem 1]{Heath}.
\begin{lemma}
For any $Q >1$, there is a positive constant $c_{Q}$, and an infinitely differentiable function $h(x,y)$ defined on the set $(0, \infty) \times \mathbb{R}$, such that
\begin{equation} \label{delatsym}
\delta(n) = \frac{c_{Q}}{Q^2} \, \sum_{c=1}^{\infty} \sideset{}{^\star}{\sum}_{ a \, \rm mod \, c} \, e\left(\frac{a \, n}{c} \right) \, h \left(\frac{c}{Q}, \frac{n}{Q^2}\right).
\end{equation}
The constant $c_{Q}$ satisfies 
$$c_{Q} = 1 + O_{N} \left(Q^{-N}\right),$$
for any $N>0$. Moreover $h(x,y) \ll x^{-1}$ for all $y$, and $h(x,y)$ is non-zero only for $x \leq \max \left(1, 2\vert y \vert\right)$.  
\end{lemma}

Also D. R.  Heath-Brown thoroughly studied the integrals  involving the function $h(x,y)$ , see sections 7 and 8 in \cite{Heath}. We also need these integral estimates in  our analysis, so we simply refer to results of these sections whenever we use them.

\section{Estimate of  $T_{q}(B)$: application of Poisson summation}
We recall that
$$T_{q}(B) = \sum_{\substack{{\bf x} \in \mathbb{Z}^4 \\ \psi_{1}({\bf x})=0}} W\left(B^{-1} {\bf x}\right) \chi_{q} \left(\psi_{2}\left({\bf x}\right)\right).$$
We now estimate $T_{q}(B)$ using circle method. First we pick the condition $\psi_{1}({\bf x}) =0$ in the definition of $T_{q}(B)$ using the delta symbol. Before writing $\delta$ symbol we use congruence equation trick to lower the ``conductor". Indeed,
\begin{align} \label{tqvalue1}
T_{q}(B) = \sum_{\substack{{\bf x} \in \mathbb{Z}^4 \\ \psi_{1}({\bf x}) \, \equiv \, 0 \, \rm mod \, q}} W\left(B^{-1} {\bf x}\right) \chi_{q} \left(\psi_{2}\left({\bf x}\right)\right) \, \delta \left(\frac{\psi_{1}\left({\bf x}\right)}{q}\right).
\end{align}
Let $Q= B / \sqrt{q}$. Using the expression \eqref{delatsym} for $\delta(n)$, we get
\begin{equation} \label{afterdelata}
T_{q}(B)= \frac{c_{Q}}{Q^2}\sum_{c=1}^{\infty} \, \sideset{}{^\star}{\sum}_{a \, \rm mod \, c} \, \sum_{\substack{{\bf x} \in \mathbb{Z}^4 \\ \psi_{1}({\bf x}) \, \equiv \, 0 \, \rm mod \, q}} W\left(B^{-1} {\bf x}\right) \chi_{q} \left(\psi_{2}\left({\bf x}\right)\right) \, e \left(\frac{a \psi_{1}({\bf x}) }{qc}\right) \, h\left(\frac{c}{Q}, \frac{\psi_{1}(\bf x)}{q Q^2}\right).
\end{equation}

By an application of Poisson summation formula we get the following lemma .
\begin{lemma}
 For any $A>0$, We have
$$T_{q}(B) = \frac{B}{q^3}\sum_{{\bf w} \in \mathbb{Z}^4} \sum_{c=1}^{\infty} \frac{1}{c^4} \, S_{q,c}\left({\bf w}\right) \, I_{q,c}\left({\bf w}\right) + O_{A}\left(B^{-A}\right),$$
where
$$S_{q,c} \left({\bf w}\right) =  \sideset{}{^\star}{\sum}_{a \, \rm mod \, c}  \, \sum_{\substack{{\bf k} \,  \rm mod \, qc \\ \psi_{1}({\bf k}) \equiv \, 0 \,  \rm mod \, q}} \chi_{q} \left(\psi_{2} \left({\bf k}\right) \right) \, e \left(\frac{a \psi_{1}({\bf k}) + {\bf w}.{\bf k}}{qc}\right),$$
and 
$$I_{q,c}\left({\bf w}\right)= \int_{\mathbb{R}^4} W\left( {\bf y}\right) \, h\left( \frac{c}{Q}, \psi_{1}({\bf y})\right) \, e \left(-\frac{B \,  {\bf w}.{\bf y}}{qc}\right) \, \mathrm{d} {\bf y}.$$
\end{lemma}

\begin{proof}
 
We now consider the inner sum over ${\bf x}$ in \eqref{afterdelata}. We split this sum into according to their residue classes modulo $qc$. Indeed, 
\begin{equation} \label{arangementmod}
\sum_{\substack{{\bf k} \,  \rm mod \, qc \\ \psi_{1}({\bf k}) \, \equiv \, 0 \, \rm mod \, q}} \chi_{q} \left(\psi_{2} \left({\bf k}\right) \right) \, e \left(\frac{a \psi_{1}({\bf k}) }{qc}\right) \, \sum_{{\bf v} \in \mathbb{Z}^4} f({\bf v}),
\end{equation}
where
$$f({\bf v}) = W\left(B^{-1} \left({\bf k} + qc {\bf v}\right)\right) \, h\left(\frac{c}{Q}, \frac{\psi_{1}\left({\bf k} +qc {\bf v}\right)}{B^2}\right).$$

We now apply the Poisson summation formula to the sum over ${\bf v}$ in \eqref{arangementmod}, then we get 
\begin{equation} \label{poisson}
\sum_{{\bf v} \in \mathbb{Z}^4} f({\bf v}) = \sum_{{\bf w} \in \mathbb{Z}^4} \widehat {f} ({\bf w}),
\end{equation}
where
$$\widehat {f} ({\bf w}) =  \int_{\mathbb{R}^4}\, f({\bf y}) \, e\left(-{\bf w}. {\bf y} \right)\mathrm{d} {\bf y}.$$

By the change of variable ${\bf z} = B^{-1} \left({\bf k} + qc \, {\bf y} \right) $, we see that 
\begin{equation} \label{fouriervalue}
\widehat {f} ({\bf w}) = \frac{B^4 }{q^4 c^4} \, e\left(\frac{{\bf k}.{\bf w}}{qc}\right) \, \int_{\mathbb{R}^4} W\left( {\bf z}\right) \, h\left( \frac{c}{Q}, \psi_{1}({\bf z})\right) \, e \left(-\frac{B \,  {\bf w}.{\bf z}}{qc}\right) \, \mathrm{d} {\bf z}.
\end{equation}

Hence,the statement of the lemma follows from  \eqref{fouriervalue}, \eqref{poisson}, \eqref{arangementmod} and \eqref{afterdelata} by noting the value $c_{Q} = 1 + O_{A}(Q^{-A})$ for any $A>0$. 
\end{proof}

\section{The Character sum}
In this section we obtain estimates on the character sum 
\begin{equation} \label{charsum}
S_{q,c} \left({\bf w}\right) =  \sideset{}{^\star}{\sum}_{a \, \rm mod \, c}  \, \sum_{\substack{{\bf k} \,  \rm mod \, qc \\ \psi_{1}({\bf k}) \equiv \, 0 \,  \rm mod \, q}} \chi_{q} \left(\psi_{2} \left({\bf k}\right) \right) \, e \left(\frac{a \psi_{1}({\bf k}) + {\bf w}.{\bf k}}{qc}\right).
\end{equation}
We have the following lemma, which shows multiplicative property of the character sum. 
\begin{lemma} \label{multi}
For $q=q_{1}q_{2}, c= c_{1}c_{2}$ with gcd $(q_{1}c_{1},q_{2}c_{2})=1$, we have
$$S_{q,c}({\bf w}) = S_{q_{1},c_{1}}({\bf w}) \, S_{q_{2},c_{2}}({\bf w}).$$
\end{lemma}
\begin{proof}
The lemma follows by writing 
$$a = a_{1} \left(q_{2}c_{2}\right) \overline{\left(q_{2}c_{2}\right)} + a_{2} \left(q_{1}c_{1}\right) \overline{\left(q_{1}c_{1}\right)} \quad \, \text{with} \quad \, a_{1} \rm \, mod \, c_{1}, \,  a_{2} \, \rm mod \, c_{2}$$
and
$${\bf k} = {\bf k}_{1} \left(q_{2}c_{2}\right) \overline{\left(q_{2}c_{2}\right)} + {\bf k}_{2} \left(q_{1}c_{1}\right) \overline{\left(q_{1}c_{1}\right)} \quad \, \text{with} \quad \, {\bf k}_{1} \rm \, mod \, q_{1} c_{1}, \, {\bf k}_{2} \, \rm mod \, q_{2} c_{2}$$
in \eqref{charsum}
and then rearranging the sum. 
\end{proof}

In the light of above Lemma \ref{multi}, we only need to study character sums of the forms
$$S_{p,1}\left({\bf w}\right), S_{p,p^r}\left({\bf w}\right) \,  \text{and} \, S_{1,p^r}\left({\bf w}\right)$$
as we are interested in for square free $q$. In the remaining section we estimate these character sums.

We have a basic result about quadratic character sum  below. Before stating  we keep some notations: let $\phi({\bf x}) = \sum_{i=1}^{4} \gamma_{i} x_{i}^2$ be the quadratic form in four variables. Let $\gamma= \prod_{i=1}^{4} \gamma_{i}$, and define associated quadratic form
$$\tilde{\phi}({\bf x}) = \sum_{i=1}^{4} \frac{\gamma \, x_{i}^2}{\gamma_{i}}.$$
We have the following result.
\begin{lemma} \label{quadaricsum}
Let $p$ be a prime with $p \nmid 2 \prod_{i=1}^{4} \gamma_{i}$. For any $m$ with $(p,m)=1$, we have
\[ 
\sum_{{\bf k} \, \rm mod \, p^r} e_{p^r} \left(m \phi({\bf k}) + {\bf k}.{\bf w} \right)=
\begin{cases} 
p^{2r} e_{p^r} \left(- \overline{4 m \gamma} \, \tilde{\phi}({\bf w})\right) & \text{if} \, \, r \, \text{is even} \\
p^{2r}  \chi_{p}(\gamma) \, e_{p^r} \left(- \overline{4 m \gamma} \, \tilde{\phi}({\bf w})\right) & \text{if} \, \, r \, \text{is odd} .
\end{cases}
\]

\end{lemma}
\begin{proof}
Let $T$ denote the sum appearing on the left-hand side of the above formula. Then we have
\begin{equation} \label{gausesum1}
T = \prod_{i=1}^{4} \left( \sum_{k \, \rm mod \, p^r} e_{p^r} \left(m \gamma_{i} k +  w_{i} k \right) \right).
\end{equation}
We can evaluate the sum inside the bracket by completing the square. Indeed, we have
$$ \sum_{k \, \rm mod \, p^r} e_{p^r} \left(m \gamma_{i} k^2 +  w_{i} k \right) = e_{p^r} \left(- \overline{4 m \gamma_{i}} w_{i}^2\right) \, \sum_{k \, \rm mod \, p^r} e_{p^r} \left(m \gamma_{i} k^2 \right)$$
as $p \nmid 2 m \prod_{i=1}^{4} \gamma_{i}$. And  we can evaluate the above  quadratic Gauss sum as follows

\[ 
\sum_{k \, \rm mod \, p^r} e_{p^r} \left(m \gamma_{i} k^2 \right)=
\begin{cases} 
p^{r/2} & \text{if} \, \, r \, \text{is even} \\
p^{r/2}  \chi_{p}(m \gamma_{i}) \, \varepsilon(p) & \text{if} \, \, r \, \text{is odd},
\end{cases}
\]
where $\varepsilon(p) =1 $ if $p \equiv \, 1 \, \rm mod \, 4$, and $\varepsilon(p) = i$ if  $p \equiv \, 3 \, \rm mod \, 4$.
On substituting this value of quadratic Gauss sum in \eqref{gausesum1}, we get the assertion of the lemma.
\end{proof}

In following subsections we estimate character sums  $S_{p,1}\left({\bf w}\right), S_{p,p^r}\left({\bf w}\right) \,  \text{and} \, S_{1,p^r}\left({\bf w}\right)$.

\subsection{Estimates of $S_{p,1}$} 
In the case when $q=p$ and  $c=1$, the character is given by
\begin{equation} \label{case1}
S_{p,1}({\bf w}) = \sum_{\substack{{\bf k} \,  \rm mod \, p \\ \psi_{1}({\bf k}) \equiv \, 0 \,  \rm mod \, p}} \chi_{p} \left(\psi_{2} \left({\bf k}\right) \right) \, e \left(\frac{{\bf w}.{\bf k}}{p}\right).
\end{equation}

We have the following lemma.
\begin{lemma} \label{sp1bound}
For any prime $p$ with $p \nmid 2 \alpha \mathcal{D}$, we have 
$$S_{p,1} \ll p^{\frac{3}{2}},$$
where implied constant depends only on the forms $\psi_{1}$ and $\psi_{2}$.
\end{lemma}

\begin{proof}
We have 
$$\chi_{p}(x)= \frac{1}{\tau(p)} \sideset{}{^\star}{\sum}_{m \, \rm mod \, p} \chi_{p}(m) \, e_{p} \left(mx\right),$$
where $\tau(p)$ is the Gauss  sum associated with the quadratic residue character $\chi_{p}$. We write this expression of $\chi_{p}$ in \eqref{case1}, then we get 
$$S_{p,1}({\bf w}) = \frac{1}{\tau(p)} \sum_{m \, \rm mod \, p} \chi_{p}(m) \, \sum_{\substack{{\bf k} \,  \rm mod \, p \\ \psi_{1}({\bf k}) \equiv \, 0 \,  \rm mod \, p}}  \, e \left(\frac{{m \psi_{2}({\bf k})+ {\bf w}.{\bf k}}}{p}\right).$$
We now pick the condition $\psi_{1}({\bf k}) \equiv \, 0 \,  \rm mod \, p$ in the inner sum of above equation using the additive character. Indeed,
\begin{equation} \label{innersumk}
S_{p,1}({\bf w}) = \frac{1}{p \tau(p)} \sum_{m \, \rm mod \, p} \chi_{p}(m) \, \sum_{\ell \, \rm mod \, p } \, \sum_{\substack{{\bf k} \,  \rm mod \, p }}  \, e \left(\frac{{m \left(\psi_{2}({\bf k}) + \ell \psi_{1}({\bf k}) \right)+ {\bf w}.{\bf k}}}{p}\right).
\end{equation}
Let $\psi_{\ell}({\bf w}) = \psi_{2}({\bf k}) + \ell \psi_{1}({\bf k})$. Thus we have $\psi_{\ell}({\bf w}) = \sum_{i=1}^{4} c_{i} w_{i}^2$, where $c_{i}= b_{i}+ \ell a_{i}$ for $i=1,2,3,4$. We set $\gamma = \prod_{i=1}^{4}c_{i}$. For  primes $p$ with $p \nmid 2 \gamma$, the innermost sum in \eqref{innersumk}, by Lemma \ref{quadaricsum}, is given by
$$\sum_{\substack{{\bf k} \,  \rm mod \, p }}  \, e \left(\frac{{m \left(\psi_{2}({\bf k}) + \ell \psi_{1}({\bf k}) \right)+ {\bf w}.{\bf k}}}{p}\right) = \chi_{p}(\gamma) \,p^2 \, e\left( \frac{- \overline{4 m \gamma} \, \tilde{\psi}_{\ell} (\bf w)}{p}\right).$$
On the other hand if $p \mid c_{i_{0}}$ for some $i_{0}$, then $\ell \equiv \, - b_{i_{0}} a_{i_{0}}^{-1} \, \rm mod \, p$ and $p \mid w_{i_{0}}$ otherwise the sum vanishes. By our assumption on prime $p \nmid 2 \alpha \mathcal{D}$, there is only one such $i$ for which $p \mid c_{i}$. So we can apply Lemma \ref{quadaricsum} to evaluate the above sum over $m$, we get that each $\ell \equiv \, - b_{i_{0}} a_{i_{0}}^{-1} \, \rm mod \, p$ contributes at most $p^{3/2}$ to $S_{p,1}$. Hence we get
\begin{align*}
S_{p,1}({\bf w}) & = \frac{p}{\tau (p)} \sum_{\substack{\ell \, \rm mod \, p \\ \ell \not \equiv - b_{i} a_{i}^{-1} \, \, \forall i}} \sum_{m \, \rm mod \, p} \chi_{p}(m \gamma) \, e\left(\frac{- \overline{4 m \gamma} \tilde{\psi}_{\ell} ({\bf w}) }{p}\right) + O\left(p^{3/2}\right) \\
& = p \, \sum_{\substack{\ell \, \rm mod \, p \\ \ell \not \equiv - b_{i} a_{i}^{-1} \, \, \forall i}} \, \chi_{p}\left(- \tilde{\psi}_{\ell} ({\bf w})\right) + O\left(p^{3/2}\right) \\
&\ll p^{3/2},
\end{align*}
where last inequality follows from the Weil bound.
\end{proof}

\subsection{Estimates of $S_{p,p^{r}}$} In this case the character sum given as 
\begin{equation} \label{case2}
S_{p,p^{r}} ({\bf w}) = \sideset{}{^\star}{\sum}_{a \, \rm mod \, p^{r}}  \, \sum_{\substack{{\bf k} \,  \rm mod \, p^{r+1} \\ \psi_{1}({\bf k}) \equiv \, 0 \,  \rm mod \, p}} \chi_{p} \left(\psi_{2} \left({\bf k}\right) \right) \, e \left(\frac{a \psi_{1}({\bf k}) + {\bf w}.{\bf k}}{p^{r+1}}\right).
\end{equation}
In the following lemma gives a bound on this character sum.
\begin{lemma} \label{spprbound}
 Let $r \geq 1$. For any $p \nmid 2 \alpha$ we have 
$$S_{p,p^{r}} ({\bf w}) \ll p^{2r +\frac{3}{2}} \,  (p^{r},\tilde{\psi}_{1}({\bf w}) ).$$
\end{lemma}

\begin{proof}
Writing the expression for $\chi_{p}$ and detecting the equation $\psi_{1}({\bf k}) \equiv \, 0 \,  \rm mod \, p$ using additive characters in $S_{p,p^{r}} ({\bf w})$, in as in the previous case, we get that
\begin{equation} \label{sprvalue}
S_{p,p^{r}}({\bf w}) = \frac{1}{p \tau(p)} \sum_{m \, \rm mod \, p} \chi_{p}(m) \, \sum_{\ell \, \rm mod \, p } \, \sideset{}{^\star}{\sum}_{a \, \rm mod \, p^{r}} \, \sum_{\substack{{\bf k} \,  \rm mod \, p^{r+1} }}  \, e \left( \frac{\psi_{a,m,\ell}({\bf k}) + {\bf k}.{\bf w}}{p^{r+1}}\right),
\end{equation}
where $\psi_{a,m,\ell}({\bf k}) = a \, \psi_{1}({\bf k}) + m p^{r} \left(\psi_{2}({\bf k}) + \ell \psi_{1}({\bf k}) \right)$. We note that if we replace $a$ by $a+x p^{r}$ with $x \, \rm mod p$ in $\psi_{a,m,\ell}({\bf k})$, then the value of 
$$e\left(\frac{\psi_{a,m,\ell}({\bf k}) + {\bf k}.{\bf w}}{p^{r+1}}\right)$$
does not change. So let  $\psi_{a ,m,\ell,x}({\bf k}) =\psi_{a+x p^{r},m,\ell}({\bf k})$. We write $\psi_{a,m,\ell,x}({\bf k}) = \sum_{i=1}^{4} c_{i} k_{i}^2$, where 
$$c_{i}= a a_{i} + m p^{r} \left(b_{i} + \ell a_{i}\right) + x p^{r} a_{i},$$
note that $c_{i}$ invertible modulo $p$ as $p \nmid 2 \alpha$, and its inverse modulo $p$ is given by
$$\overline{c}_{i} = \overline{a a_{i}} - \overline{a a_{i}}^2 \, m p^{r} \left(b_{i} + \ell a_{i}\right) - \overline{a }^2   x p^{r} \overline{  a_{i}} $$
for $i=1,2,3$ and $4$. By Lemma \ref{quadaricsum} we get
$$\sum_{\substack{{\bf k} \,  \rm mod \, p^{r+1} }}  \, e \left( \frac{\psi_{a,m,\ell}({\bf k}) + {\bf k}.{\bf w}}{p^{r+1}}\right) = \chi_{p}(\alpha)^{r+1} p^{2r+2} \, e\left(\frac{- \overline{4 \gamma} \,  \tilde{\psi}_{a,m,\ell} ({\bf w})}{p^{r+1}}\right),$$
where $\gamma= \prod_{i=1}^{4} c_{i}$. 

\vspace{.5cm}
Now we see that
\begin{align*}
&e\left(\frac{- \overline{4 \gamma} \,  \tilde{\psi}_{a,m,\ell} ({\bf w})}{p^{r+1}}\right)  = e\left(\frac{- \bar{4} \, \sum_{i=1}^{4} \overline{ c_{i}} w_{i}^2}{p^{r+1}}\right) \\
& = e\left(\frac{- \bar{4} \bar{a} \, \sum_{i=1}^{4} \overline{ a_{i}} w_{i}^2}{p^{r+1}}\right) \, e\left(\frac{ \bar{4} m \overline{a}^{2}  \left(\sum_{i=1}^{4} \overline{  a_{i}}^{2} b_{i} w_{i}^2 +  \ell \sum_{i=1}^{4} \overline{  a_{i}}^{2} a_{i} w_{i}^2\right) }{p}\right) e\left(\frac{ x \, \bar{4} \bar{a}^2 \, \sum_{i=1}^{4} \overline{ a_{i}} w_{i}^2}{p}\right)
\end{align*}
by using the value of $\overline{c_{i}}$. The last factor above, involving $x$, tells us that the sum over ${\bf k}$ is zero unless $p \mid \tilde{\psi}_{1}({\bf w})$. keeping above things in mind, first interchanging sum over $a$ and $m$ and changing variable $m \mapsto m a^2$ and then execute sum over $a$ in the expression of $S_{p,p^{r}}$ given in \eqref{sprvalue}. Then we get
\begin{align*}
S_{p,p^{r}}({\bf w}) = &p^{2r+1} \, \chi_{p}(\alpha)^{r+1} \,   h_{p^{r+1}}\left(\tilde{\psi}_{1}(\bf w)\right)  \sum_{\ell \, \rm mod \, p} \, \\
 & \times \sum_{\ell \, \rm mod \, p} \, \frac{1}{\tau(p)} \sum_{m \, \rm mod \, p} \chi_{p}(m) \,e\left(\frac{ \bar{4} m   \left(\sum_{i=1}^{4} \overline{  a_{i}}^{2} b_{i} w_{i}^2 +  \ell \sum_{i=1}^{4} \overline{  a_{i}}^{2} a_{i} w_{i}^2\right) }{p}\right) \\
 &= p^{2r+1} \, \chi_{p}(\alpha)^{r+1} \,   h_{p^{r}}\left(\frac{\tilde{\psi}_{1}(\bf w)}{p}\right)  \sum_{\ell \, \rm mod \, p} \, \chi_{p}\left(\sum_{i=1}^{4} \overline{  a_{i}}^{2} b_{i} w_{i}^2 +  \ell \sum_{i=1}^{4} \overline{  a_{i}}^{2} a_{i} w_{i}^2\right) \\
 & \ll p^{2r+\frac{3}{2}} \, \left(p^{r},\tilde{\psi}_{1}({\bf w})/p\right), 
\end{align*}
in the last inequality we used bound on Ramanujan sum and  the Weil bound for the character sum.  Hence, the lemma follows.
\end{proof}

\subsection{Estimates of $S_{1,p^{r}}$}
In this case the character sum is given by
$$S_{1,p^{r}} ({\bf w}) = \sideset{}{^\star}{\sum}_{a \, \rm mod \, p^{r}}  \, \sum_{\substack{{\bf k} \,  \rm mod \, p^{r} }} \, e \left(\frac{a \psi_{1}({\bf k}) + {\bf w}.{\bf k}}{p^{r}}\right).$$
Explicit evaluation of this sums studied in \cite{H-R}. 
\begin{lemma}  \cite[Lemma 4.4]{H-R}
For any prime $p \nmid 2 \alpha$, we have
$$S_{1,p^{r}} ({\bf w})= \chi_{p}(\alpha)^{r} \, h_{p^{r}}(\tilde{\psi}_{1}({\bf w})) \, p^{2r}.$$
\end{lemma}

From this lemma and standard properties of the Ramanujan sum we can deduce following corollaries.
\begin{corollary} \label{s1cbound1}
For any prime $p \nmid 2 \alpha$, we have 
\[
S_{1,p^{r}}=
\begin{cases}
\chi_{p}(\alpha) \, p^2 (p-1) & \text{if} \, p \mid \tilde{\psi}_{1}({\bf w}) \\
- \chi_{p}(\alpha) p^2 & p \nmid \tilde{\psi}_{1}({\bf w}).
\end{cases}
\]
\end{corollary}

\begin{corollary} \label{s1cbound}
We have 
$$S_{1,c}({\bf w}) \ll c^{2} (c, \tilde{\psi}_{1}({\bf w})),$$
where implied constants depending on the forms $\psi_{1}$ and $\psi_{2}$.
\end{corollary}

\subsection{Cancellation in the sum of $S_{q,c}$}

We have the following lemma.
\begin{lemma} \label{averagemodvalue}
Let $(q,2 \alpha \mathcal{D})=1$. Then for any ${\bf w} \in \mathbb{Z}^4$ with $\tilde{\psi}_{1}({\bf w}) \neq 0$, we have
$$\sum_{\substack{c \leq X \\ q \mid c}} \vert S_{q,c}({\bf w})\vert \ll q^{3/2} \, X^{3} \, \vert q \tilde{\psi}_{1}({\bf w})\vert^{\epsilon}, $$
where implied constants depends only on the forms $\psi_{i}$ and $\epsilon$.
\end{lemma}

\begin{proof}
We first write $c = q_{\infty} c^{\prime}$ with $(q,c^{\prime})=1$,  where $q_{\infty} = \prod_{p \mid q, p^{r_{i}} || c} p^{r_{i}}$. Then by multiplicativity  of $S_{q,c}$ we get
\begin{align*}
S_{q,c} ({\bf w}) &= S_{q,q_{\infty}} ({\bf w}) \, S_{1,c^{\prime}} ({\bf w}) \\
&= \prod_{\substack{p \mid q \\ p^{r_{i}} || c} } S_{p, p^{r_{i}}} ({\bf w}) \,S_{1,c^{\prime}} ({\bf w}). 
\end{align*}
We now use Corollary \ref{s1cbound} and Lemma \ref{spprbound} to bound $S_{q,c}$,  when $q \mid c$. Indeed,

\begin{align*}
S_{q,c} ({\bf w}) & \ll \prod_{\substack{p \mid q \\ p^{r_{i}} || c} } p^{2 r_{i} +3/2} (p^{r_{i}}, \tilde{\psi}_{1}({\bf w})) \, {c^{\prime}}^{2} (c^{\prime}, \tilde{\psi}_{1}({\bf w})) \\
& = q^{3/2} c^{2} (q_{\infty} , \tilde{\psi}_{1}({\bf w})) \, (c^{\prime}, \tilde{\psi}_{1}({\bf w})) \\
& \leq  q^{3/2} c^{2} (c  , \tilde{\psi}_{1}({\bf w})),
\end{align*}
in the equality above  we use the fact that $q\mid c$. We now estimate the sum given in the lemma. Using the above bound on $S_{q,c}({\bf w})$, we get
\begin{align*}
\sum_{\substack{c \leq X \\ q \mid c}} \vert S_{q,c}({\bf w})\vert & \ll q^{3/2} \sum_{\substack{c \leq X \\ q \mid c}}   c^{2} (c, \tilde{\psi}_{1}({\bf w})) \\
& \ll q^{3/2} \sum_{d \mid \tilde{\psi}_{1}({\bf w}) } d^3 \sum_{c \leq \frac{X}{d}} l^2 \\
&\ll q^{3/2} X^{3} \vert  \tilde{\psi}_{1}({\bf w})\vert^{\epsilon}.
\end{align*}
Hence the lemma follows.
\end{proof}

The proof of the following lemma  closely follows that  of the Lemma 5.2 in \cite{H-R}.
\begin{lemma} \label{cancelationsum}
Suppose $q=q_{1}q_{2}$ is square free, $(q,2 \alpha \mathcal{D}) =1$ and $\alpha =a_{1} a_{2}a_{3}a_{4}$ is not square. Then for ${\bf w} \in \mathbb{Z}^{4}$ with $\tilde{\psi}_{1}({\bf w}) = 0$, we have
$$\sum_{\substack{c \leq X \\ (c,q)=q_{1}}} S_{q_{1},c} ({\bf w})  \, \ll \, q_{1}^{3/2} X^{7/2} \, (Bq)^{\epsilon},$$
where implied constants depends only on the forms $\psi_{1}, \psi_{2}$ and $\epsilon$.
\end{lemma}

\begin{proof}
To estimate the partial sum  $\sum_{\substack{c \leq X \\ (c,q)=q_{1}}} S_{q_{1},c} ({\bf w})$,  we use the Perron's formula. For this we need to have analytic properties of its Dirichlet series 
$$\sum_{\substack{c=1 \\ (c,q)=q_{1}}}^{\infty} S_{q_{1},c} ({\bf w}) \, c^{-s}.$$
By multiplicative property of the character sum $S_{q,c}({\bf w})$, given in Lemma \ref{multi}, we see that
$$\sum_{\substack{c=1 \\ (c,q)=q_{1}}}^{\infty} S_{q_{1},c} ({\bf w}) \, c^{-s} = \prod_{p \mid q_{1}} \left(S_{p,p}({\bf w}) p^{-s} + S_{p,p^2}({\bf w}) p^{-2s} + \ldots\right) \, F_{q}(s;{\bf w})$$
where
$$F_{q}(s;{\bf w}) = \sum_{\substack{c=1 \\ (c,q)=1}}^{\infty} S_{1,c} ({\bf w}) \, c^{-s} = \prod_{p \nmid q} \left(\sum_{n=0}^{\infty} S_{1,p^{n} }({\bf w}) \,  p^{-ns}\right).$$
It follows from Corollary \ref{s1cbound} that $F_{q}(s;{\bf w})$ converges absolutely for $\sigma > 4$, and from Lemma \ref{spprbound} it follows that the product over $p \mid q_{1}$ is analytic for $\sigma > 7/2$ and is bounded by $O(q_{1}^{3/2})$. For prime $p \nmid 2 \alpha$, by Corollaries \ref{s1cbound1} and \ref{s1cbound} we get
$$\sum_{n=0}^{\infty} S_{1,p^{n} }({\bf w}) \,  p^{-ns} = 1 + p^{2} (p-1) \chi_{p}(\alpha) p^{-s} + O(p^{-1-\delta}),$$
for $\sigma \geq \frac{7+\eta}{2}$. Therefore in this half plane we can write 
$$F_{q}(s;{\bf w}) = L(s-3,\chi_{p}) \, f(s; {\bf w}) \, \prod_{p \mid q} \left(\sum_{n=0}^{\infty} S_{1,p^{n} }({\bf w}) \,  p^{-ns}\right),$$
where $f(s; {\bf w})$ is a Dirichlet series which is convergent absolutely for $\sigma > 7/2$.  Hence $F_{q}(s;{\bf w})$ has a  analytic continuation to the half plane $\sigma > 7/2$.  Also in this plane we have the bound
$$f(s; {\bf w}) \, \prod_{p \mid q} \left(\sum_{n=0}^{\infty} S_{1,p^{n} }({\bf w}) \,  p^{-ns}\right) \ll_{\eta} d(q),$$
for $\sigma > (7+ \eta )/2$. The lemma follows from the Perron's formula, and observing that $L(s-3, \chi_{p})$ does not have a pole at $s=4$ as $\alpha$ is not a square number.

\end{proof}

\section{estimation of the integral $I_{q,c}$}

We recall the integral 
$$I_{q,c}\left({\bf w}\right)= \int_{\mathbb{R}^4} W\left( {\bf y}\right) \, h\left( \frac{c}{Q}, \psi_{1}({\bf y})\right) \, e \left(-\frac{B \,  {\bf w}.{\bf y}}{qc}\right) \, \mathrm{d} {\bf y}.$$
This integral vanishes unless $c \ll Q$ as the function $h(x,y)=0$ unless $x \leq \max (1,2|y|)$.
In the notation of Heath-Brown \cite[Sections 7 and 8]{Heath}, we have
$$I_{q,c}\left({\bf w}\right) = I_{r}^{*}({\bf v}) \quad \text{with} \, \, r= \frac{c}{Q}, \, {\bf v} = \frac{B {\bf w}}{qQ}.$$

The following lemma, follows from  Lemma 14 and Lemma 18 in \cite[Sections 7 and 8]{Heath}.

\begin{lemma} \label{sizeofw}
For ${\bf w} \neq 0$ and for any $N \geq 0$, we have
$$I_{q,c}({w}) \ll_{N} \, \frac{Q}{c} \left( \frac{qQ}{B \vert {\bf w} \vert} \right)^{N}. $$
\end{lemma}
By the above lemma, the integral $I_{q,c}$ is negligibly small unless 
$$ \vert {\bf w} \vert \ll \, \frac{qQ}{B} \, B^{\epsilon} = \sqrt{q} B^{\epsilon}.$$ 

In this case we have a bounds on $I_{q,c}$ which are given in the following lemma, which follows from  Lemma 14 and Lemma 22 in \cite[Sections 7 and 8]{Heath}.
\begin{lemma} \label{integralesti}
For $0 < \vert {\bf w}\vert \leq \, \sqrt{q} B^{\epsilon}$, we have
\begin{enumerate}
	\item $I_{q,q_{1} t} ({\bf w}) \ll \, \frac{q_{1}^2 t q_{2}}{B} \, \vert {\bf w} \vert^{-1} \, B^{\epsilon}$ and
	\item $ \frac{d}{dt} I_{q,q_{1} t} ({\bf w}) \ll \, \frac{q_{1}^2  q_{2}}{B} \, \vert {\bf w} \vert^{-1} \, B^{\epsilon}$.
\end{enumerate}
\end{lemma}

The following lemma gives bound on the integral $I_{q,c}(\bf 0)$ in the case of zero frequency. 
\begin{lemma}
	We have $$I_{q,c}({\bf 0}) \ll 1.$$
\end{lemma} 
The proof of this lemma follows from the  Lemma 13 in \cite[Sections 7]{Heath}.
\section{Proof of the Proposition \ref{propontq}} \label{tqbound}
We need to estimate 
$$T_{q}^{*}(B) = \frac{B^2}{q^3}\sum_{{\bf w} \in \mathbb{Z}^4} \sum_{c=1}^{\infty} \frac{1}{c^4} \, S_{q,c}\left({\bf w}\right) \, I_{q,c}\left({\bf w}\right).$$
We have $c \ll Q= B/\sqrt{q}$ otherwise the integral $I_{q,c}$ vanishes. We also have, by Lemma \ref{sizeofw}, that 
$$\vert {\bf w} \vert \ll \, \frac{qQ}{B} \, B^{\epsilon} = \sqrt{q} B^{\epsilon}.$$

\subsection{Zero frequency: {\bf w}=0}
The contribution to $T_{q}(B)$ from the zero frequency is 
$$ \frac{B^2}{q^{3}} \sum_{c=1}^{\infty} \frac{1}{c^4} \, S_{q,c}\left({\bf 0}\right) \, I_{q,c}\left({\bf 0}\right).$$
And this is 
$$ \ll \frac{B^2}{q^{3}} \sum_{c \ll B/\sqrt{q}} \frac{1}{c^4} \, \vert S_{q,c}({\bf 0})\vert $$
as $I_{q,c}({\bf 0}) \ll 1$. By Corollary \ref{s1cbound} and Lemma \ref{spprbound}, we have
$$S_{q,c}({\bf 0}) \ll q^{3/2} \, c^{3}.$$
Thus, the zero frequency contribution is at most
$$ \frac{B^2}{q^{3}} \sum_{c \ll B/\sqrt{q}} \frac{1}{c^4} \, \vert S_{q,c}({\bf 0})\vert \ll \frac{B^2}{q^{3/2}}.$$

\subsection{Non-zero frequency: the case $\tilde{\psi}_{1}({\bf w}) \neq 0$}
We write 
$$J({\bf w}) = \sum_{\substack{c \leq B/\sqrt{q} \\ (c,q)=q_{1}}}\frac{1}{c^4} \, S_{q,c}\left({\bf w}\right) \, I_{q,c}\left({\bf w}\right),$$
where $q=q_{1}q_{2}$ be the product of square free numbers. We have, by Lemma \ref{integralesti} and Lemma \ref{sp1bound}, that
$$J({\bf w}) \ll \frac{q_{2}^{3/2} q}{B \, \vert {\bf w}\vert } \, \sup_{C \leq B/\sqrt{q}} \, \frac{1}{C^{3}}\, \sum_{\substack{C \leq c \leq 2C \\ (c,q)=q_{1}}}  \, \vert S_{q_{1},c}\vert .$$
We now use Lemma \ref{averagemodvalue} to estimate the sum on the right of above. Indeed, we have
$$J({\bf w}) \ll \frac{q^{5/2} }{B \vert {\bf w}\vert} \, B^{\epsilon}.$$
Therefore, the contribution to $T_{q}(B)$ coming from those ${\bf w} \neq 0$ for which $\tilde{\psi}_{1}({\bf w}) \neq 0$ is at most
$$\frac{B^{2}}{q^{3}} \sum_{\substack{0 < \vert {\bf w}\vert \ll \sqrt{q} B^{\epsilon} \\ \tilde{\psi}_{1}({\bf w}) \neq 0}} \, J({\bf w})$$
and which is bound by
$$\frac{B }{q^{1/2} } \sum_{\substack{0 < \vert {\bf w}\vert \ll \sqrt{q} B^{\epsilon} }} \frac{1}{\vert {\bf w}\vert} \ll q B B^{\epsilon} \leq q B^{1+\epsilon}$$
since 
$$\sum_{\substack{0 < \vert {\bf u}\vert \leq U} } \frac{1}{\vert {\bf U}\vert} \ll \, U^{3+\epsilon}.$$

\subsection{Non-zero frequency: the case $\tilde{\psi}_{1}({\bf w}) =0$}

We consider $J({\bf w})$ for ${\bf w} \neq 0$ with $\tilde{\psi}_{1}({\bf w}) =0$. In this case 
$$J({\bf w}) \leq \, \vert  S_{q_{2},1} ({\bf w}) \vert \, \sup _{C \leq B/\sqrt{q}} \,  \Big \vert \sum_{\substack{C \leq c \leq 2C \\ (c,q)=q_{1}}}\frac{1}{c^4} \, S_{q_{1},c}\left({\bf w}\right) \, I_{q,c}\left({\bf w}\right) \Big \vert.$$
By Lemma \ref{integralesti}, Lemma \ref{cancelationsum}, Lemma \ref{sp1bound} and partial summation we have that
$$J({\bf w}) \ll q_{2}^{3/2} \, \sup_{C \leq B/\sqrt{q} } \frac{q_{1}^{5/2} q_{2} C^{1/2}}{B \vert {\bf w}\vert } \, \leq \, \frac{q^{9/4} }{B^{1/2} \vert {\bf w} \vert}.$$
Hence, the contribution to $T_{q}(B)$ coming from those ${\bf w} \neq 0$ for which $\tilde{\psi}_{1}({\bf w}) = 0$ is at most

$$\frac{B^{2}}{q^{3}} \sum_{\substack{0 < \vert {\bf w}\vert \ll \sqrt{q} B^{\epsilon} \\ \tilde{\psi}_{1}({\bf w}) \neq 0}} \, J({\bf w}) \ll \frac{B^2}{q^{3}} \frac{q^{9/4}}{B^{1/2} } \sum_{\substack{0 < \vert {\bf w}\vert \ll \sqrt{q} B^{\epsilon} \\ \tilde{\psi}_{1}({\bf w}) =0}} \frac{1}{\vert {\bf w}\vert }.$$  
We can evaluate the last sum by breaking into dyadic blocks
$$\sum_{\substack{{\bf w} \\ W_{i} \leq {\bf w} \leq 2 W_{i} \\ \tilde{\psi}_{1}({\bf w})=0}} \frac{1}{\vert {\bf w}\vert }  \ll  \left(\max W_{i}\right)^{-1} \sum_{\substack{{\bf w} \\ W_{i} \leq {\bf w} \leq 2 W_{i} \\ \tilde{\psi}_{1}({\bf w})=0}} 1 \, \ll \,  \left(\max W_{i}\right)^{1+\epsilon},$$
where we used the definition of $\vert {\bf w}\vert = \max_{i} \, {\vert w_{i} \vert }$. Using this estimate we get that 
$$\frac{B^{2}}{q^{3}} \sum_{\substack{0 < \vert {\bf w}\vert \ll \sqrt{q} B^{\epsilon} \\ \tilde{\psi}_{1}({\bf w}) \neq 0}} \, J({\bf w}) \ll \, \frac{B^{3/2 + \epsilon}}{q^{1/4}}.$$

We conclude that
$$T_{q}(B) \ll \left( \frac{B^2}{q^{3/2}} + q B + \frac{B^{3/2}}{q^{1/4}} \right) B^{\epsilon}.$$
Hence Proposition \ref{propontq} holds.
\section{The conclusion: proof of the Proposition \ref{malpropo}} \label{final}
Recall from \eqref{mstarbound} that
$$M^{*}(B) \ll_{\epsilon} \frac{B^{2+\epsilon}}{P} + \frac{1}{P^{2}} \sum_{p_{1} \neq p_{2} \in \mathcal{P}} T_{p_{1}p_{2}}\left(B\right)$$
and $\mathcal{P}$ contains prime of size $P$. Thus for any $q = p_{1}p_{2}$ is of size $P^2$. On substituting the bound for $T_{q}(B)$ in the above inequality  we get
$$M^{*}(B) \ll_{\epsilon} \left(\frac{B^{2}}{P} +  \frac{B^2}{P^{3}} + P^{2} B + \frac{B^{3/2}}{P^{1/2}} \right) B^{\epsilon} .$$
The optimal choice for $P$ is $P=B^{1/3}$, for this choice we have the bound
$$M^{*}(B) \ll_{\epsilon} \ll B^{\frac{5}{3}+\epsilon},$$
hence Proposition \ref{malpropo} follows.

\section*{Acknowledgements}
I am thankful to Prof. Ritabrata Munshi for sharing his ideas and his support throughout the work. I also wish to thank Prof. Satadal Ganguly for his encouragement and  constant support. I am  grateful to  Stat-Math Unit, Indian Statistical Institute, Kolkata for providing wonderful research environment.

\end{document}